\documentclass[a4paper]{article}
\usepackage[utf8]{inputenc}
\usepackage{fullpage}

\title{\vspace{-1cm}A telescopic proof of Cayley's formula}
\author{Guillaume Chapuy%
\thanks{Université Paris Cité, CNRS, IRIF, F-75013, Paris, France
Email:~{\tt guillaume.chapuy@irif.fr}.
}
\and Guillem Perarnau%
\thanks{Departament de Matem\`atiques and IMTECH, Universitat Polit\`ecnica de Catalunya (UPC), Barcelona, Spain. Centre de Recerca Matemàtica, Barcelona, Spain. Email:~{\tt guillem.perarnau@upc.edu}. 
}
}

\date{\today}

\usepackage{graphicx, enumitem}
\usepackage{amsmath,amsfonts,amssymb}
\usepackage{amsthm, xcolor}
\usepackage{comment, fullpage, placeins,hyperref }

\usepackage[margin=7mm]{caption}

\theoremstyle{theorem}
\newtheorem{theorem}{Theorem}

\numberwithin{theorem}{section}

\theoremstyle{definition}

\begin{document}

\maketitle

\vspace{-10mm}

\begin{abstract} 
	We give a short proof of the fact that the number of labelled trees on $n$ vertices is $n^{n-2}$.

	Although many short proofs are known, we have not seen this one before.
\end{abstract}

\addtocounter{section}{1}

Let $f:[n]\to[n]$, and let $G_f$ be its associated directed graph, with vertex set $[n]$ and directed edges $(i,f(i))$ for $i \in [n]$. A vertex is \emph{cyclic} if it belongs to a cycle of $G_f$. Functions $f$ having a unique cyclic vertex are in bijection with rooted trees on $[n]$ (the unique cyclic vertex is the root, and edges are oriented towards the root, see Figure~\ref{fig:cayleyTelescopic}).

Cayley's formula~\cite{Borchardt, Cayley} asserts that there are $n^{n-2}$ trees on the vertex set $[n]$, or equivalently there are $n^{n-1}$ rooted trees. Therefore it is equivalent to the following statement: 
\begin{theorem}
	The probability that a uniform random function $f:[n]\to [n]$ has a unique cyclic vertex is~$\frac{1}{n}$.
\end{theorem}
\begin{proof}
We reveal the edges of $G_f$ according to the randomized procedure below. 
We call a vertex \emph{explored} when the edge outgoing from it has already been revealed.
\smallskip

\textit{\hspace{-0.5cm} Set $i:=0$. While there are unexplored vertices, do:
	\begin{itemize}[itemsep=0pt, topsep=0pt, parsep=0pt, leftmargin=25pt]
		\item[(1)] Set $i:=i+1$, and pick arbitrarily a vertex, say $V_i$, among the vertices yet unexplored;
		\item[(2)] explore the \emph{future} of $V_i$ by revealing $f(V_i), f^2(V_i), \dots$ iteratively until an explored vertex is reached.
	\end{itemize}
	}
\medskip

\noindent
We let $T_i$ be the (random) number of vertices explored after the $i$-th round of the procedure, and $K$ the total (random) number of rounds. Note that $T_K=n$.
	The random graph $G_f$ has a unique cyclic vertex if and only if (see Figure~\ref{fig:cayleyTelescopic}):

	\begin{itemize}[itemsep=0pt, topsep=0pt, parsep=0pt, leftmargin=20pt]
		\item[-] the edge revealed at time $T_1$ is a loop. Conditionally to $T_1$, this happens with probability $\frac{1}{T_1}$.
		\item[-] 
for each $i\geq 2$, the edge revealed at time $T_i$ connects to one of the $T_{i-1}$ vertices explored in previous rounds -- since connecting instead to one of the $T_{i}-T_{i-1}$ vertices explored in the present round would create a new cycle. 
		Conditionally to $(T_1,\dots,T_i)$, this happens with probability $\frac{T_{i-1}}{T_i}$.
	\end{itemize}
\noindent
	Therefore, conditionally to $(T_1,\dots,T_K;K)$, the probability that $G_f$ has a unique cyclic vertex is
	$$
	\frac{1}{T_1} \times
	\frac{T_1}{T_2}\times
	\dots \times
	\frac{T_{K-1}}{T_K}= \frac{1}{T_K} =\frac{1}{n}.
	$$
	Since the probability is independent of $(T_1,\dots,T_K;K)$ this is also true unconditionally.
\end{proof}

\vfill
\begin{figure}[h]
	\centering
	\includegraphics[height=4cm]{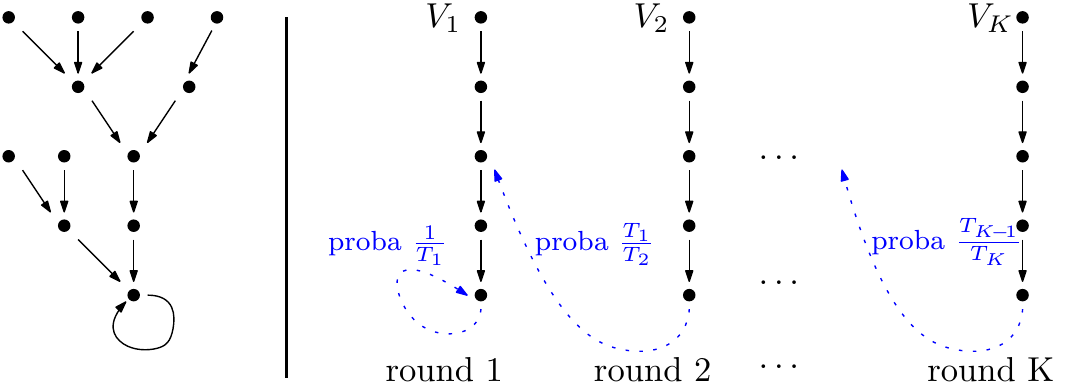}
	\caption{Left: a function $f:[n]\rightarrow [n]$ with a unique cyclic vertex (i.e. a rooted Cayley tree) with $n=12$ and labels in $[n]$ not represented; Right: short proof of Cayley's formula. 
	}
	\label{fig:cayleyTelescopic}
\end{figure} 

\newpage
\noindent \textbf{Comments and origin of the proof.} 
\smallskip

Many beautiful and/or short and/or deep proofs of Cayley's formula are known, and listing them or studying their history and relations is far beyond our aim.

%
%

We found the present proof while writing our recent paper~\cite{GCGP:synchronizing} on synchronization of random automata. Most of that paper studies $w$-trees, a certain class of 2-letter automata which are a generalization of trees. We adapt to $w$-trees the classical Joyal bijection which transforms functions into trees by edge-rewiring (see e.g.~\cite{proofsFromTheBook}), but some of the arguments also require us to construct $w$-trees by recursive explorations. For this, in \cite[Section 10]{GCGP:synchronizing},  we use a telescopic argument which is a weak adaptation to $w$-trees of the construction presented here, which we chose to write separately.





To conclude, we observe that an immediate consequence of our proof is the following well-known fact. Let $H_n$ denote the height of a uniform random vertex in a uniform random rooted tree with $n$ vertices. Then $1+H_n$ has the same law as the time of first repetition in a i.i.d. sequence of uniform variables on~$[n]$ (time of first collision in the coupon collector process). 

\bibliographystyle{alpha}
\bibliography{biblio}

\end{document}